\documentclass[a4,psfig,english,12pt,epsf,portrait]{article}
\usepackage{amsmath,amsfonts,latexsym,amscd,amssymb,theorem}
\usepackage{epsfig}
\usepackage{color}
\usepackage{amsmath}
\usepackage{graphicx}
\usepackage{bbm}
\usepackage{titlesec}
\usepackage[titletoc,toc,title]{appendix}
\def\R{\hbox{\bf R}}
\def\Z{\hbox{\bf Z}}

\def\A{{\cal A}}

\def\F{{\cal F}}

\def\b{\mathbf{b}}

\def\a{\mathbf{a}}

\def\K{{\cal K}}

\def\<{\langle}
\def\>{\rangle}
\newcommand{\ba}{\begin{eqnarray}}
\newcommand{\ea}{\end{eqnarray}}

\newcommand{\mres}{\mathbin{\vrule height 1.6ex depth 0pt width
 0.13ex\vrule height 0.13ex depth 0pt width
 1.3ex}}

\newtheorem{thm}{Theorem}[section]

\newtheorem{theorem}[thm]{Theorem}
\newtheorem{definition}[thm]{Definition}
\newtheorem{lemma}[thm]{Lemma}
\newtheorem{proposition}[thm]{Proposition}
\newtheorem{corollary}[thm]{Corollary}
\newtheorem{rem}[thm]{Remark}

\newcommand{\eps}{\epsilon}
\numberwithin{equation}{section}

\renewcommand{\R}{{\mathbb R}}
\renewcommand{\Z}{{\mathbb Z}}

\begin{document}

\title{CYCLE CHARACTERIZATION OF THE AUBRY SET FOR  WEAKLY COUPLED HAMILTON--JACOBI SYSTEMS}
\author{H.Ibrahim\thanks{Lebanese University, Mathematics Department,
Hadeth, Beirut, Lebanon,  {\tt ibrahim@cermics.enpc.fr}},
\,A.Siconolfi\thanks{Mathematics Department, University of Rome "La
Sapienza",  Italy, {\tt siconolf@mat.uniroma1.it}}, \,
S.Zabad\thanks{Mathematics Department, University of Rome "La
Sapienza",  Italy, {\tt zabad@mat.uniroma1.it}}} \maketitle


\begin{abstract}
We study a class of weakly coupled systems of Hamilton--Jacobi
equations using the random frame introduced in \cite{siconolfi3}. We
provide a cycle condition characterizing the points of
Aubry set. This  generalizes a property already known in the scalar case. \\

\bigskip

\bigskip


\end{abstract}


\section{\textbf{Introduction}}
\parskip +3pt
True to the title, the object of the paper is to provide a dynamical
characterization of the Aubry set associated to weakly coupled
Hamilton-Jacobi systems posed on the flat torus $\mathbb T^N$. We
consider the one--parameter family
$$
H_{i}(x,Du_{i})+\sum_{j=1}^{m}a_{ij}u_{j}(x)=\alpha \quad \mbox{in }
\mathbb{T}^{N} \quad \mbox{ for every } i \in \{1,\cdots,m\},
$$
with  $m \ge2$ and $\alpha$ varying in $\R$. Here $\mathbf
u=(u_1,\cdots,u_m)$ is the unknown function, $H_1, \cdots, H_m$ are
continuous Hamiltonians, convex and superlinear in the momentum
variable, and $A = \left (a_{ij} \right )$ is a coupling matrix.

As primarily pointed out in  \cite{Fabio}, \cite{Mitake2}, this kind
of systems exhibit properties and phenomena similar to the ones
already studied for scalar Eikonal equations,  under suitable
assumptions on the coupling matrix.

In particular   the minimum $\alpha$ for which the system has
subsolutions, to be understood in viscosity or equivalently a.e.
sense, is the unique value for which it admits viscosity solutions.
We call this threshold value  critical, and  denote it  by $\beta$
in what follows.

The obstruction to the existence of subsolutions below the critical
value  is not spread indistinctly on the torus, but instead
concentrated on the Aubry set,  denoted by $\A$. This fact, proved
in \cite{Davini}, generalizes what happens in the scalar case. It
 has relevant consequences on the structure of the critical
subsolutions and allow defining a fundamental class of critical
solutions, see Definition \ref{pdeau}. However, so far, no
geometrical/dynamical description of $\A$ is available, and the aim
of our   investigation is precisely to mend this gap.

To deepen knowledge of the Aubry set  seems important  for the
understanding of  the interplay between equational and dynamical
facts in the study of the system, which  is at the core of an
adapted weak KAM theory. See \cite{Fathh} for a comprehensive
treatment of this topic in the scalar case. This will hopefully
allow to attack some open problem in the field, the most relevant
being the existence of regular subsolutions. Another related
application, at least when the Hamiltonians are of Tonelli type, is
in the analysis of random evolutions associated to weakly coupled
systems, see \cite{DaZaSi}.

 According to \cite{Davini}, there is  a restriction in the values that a critical (or
supercritical)  subsolution can assume at any given point. This is a
property which genuinely depends on the vectorial structure
  of the problem and has no counterpart in the scalar case.
  Due to stability properties of viscosity
subsolutions and the convex nature of the problem, these admissible
values make up a closed convex set at any point $y$ of the torus. We
denote it  by $F_\beta(y)$.

The restriction becomes severe on the Aubry set where $F_\beta(y)$
is a one--dimensional set, while we have proved, to complete the
picture, that it possesses nonempty interior outside $\mathcal A$,
see Proposition \ref{prop30}. In a nutshell what we are doing in the
paper is to provide a dynamical dressing to this striking dichotomy.

To this purpose, we take advantage of the action functional
introduced  in \cite{siconolfi3} in  relation to the systems. We
also
 make a crucial use of  the characterization of admissible values
through the action functional computed on random cycles there
established, see Theorem \ref{t18}.

The action functional is defined exploiting the underlying random
structure given by the Markov chain with $-A$ as transition matrix.
Following the approach of \cite{siconolfi3}, we provide a
presentation of the random frame based on explicit computations and
avoid using advanced probabilistic notions.  This makes  the text
mostly self contained accessible to readers  without specific
background in probability.

The starting point is the cycle characterization of the Aubry set
holding in the scalar case, see \cite{siconolfi}. It asserts that a
point is in the Aubry set if and only there exists, for some $\eps$
positive,  a sequence of cycles based on it, and  defined in $[0,t]$
with $t > \eps$, on which the action functional is infinitesimal. Of
course the role of the lower bound $\eps$ is crucial, otherwise the
property should be trivially true for any element of the torus.

To generalize it in the context of systems,  we need using random
cycles defined on intervals with a stopping time, say $\tau$,  as
right endpoint. We call it $\tau$--cycles, see Appendix
\ref{randomset}. This makes the adaptation of the  $\eps$--condition
quite painful. To perform the task, we use the notion of stopping
time strictly greater than $\eps$, $\tau \gg \eps$, see Definition
\ref{superstrict}, which seems rather natural but that we were not
able to find in the literature. We therefore present  in Section
\ref{propstoptim} some related basic results. We, in particular,
prove that the exponential of the coupling matrix related to a $\tau
\gg \eps$ is strictly positive, see Proposition \ref{corollozzo}.
This property will be repeatedly used throughout the paper.

We moreover provide  a strengthened version of the aforementioned
Theorem \ref{t18}, roughly speaking showing that the $\tau$--cycles
with $\tau \gg \eps$ are enough to characterize admissible values
for critical subsolutions, see Theorem \ref{t20}. This result is in
turn based on a cycle iteration technique we explain in Section
\ref{reptcyc}.

\smallskip

The main output is presented in two versions, see Theorems
 \ref{t19}, \ref{t21}, with the latter one, somehow
more geometrically flavored, exploiting the notion of characteristic
vector of a stopping time, see Definition \ref{charactsc}.

The paper is organized as follows: in section \ref{settingpb} we
introduce the system under study and recall
 some basic preliminary facts. Section \ref{propstoptim} is devoted to illustrate  some  properties of stopping times and
 the related shift flows. Section \ref{reptcyc} is about the cycle iteration
 technique.
 In section \ref{dynpropaub} we give  the main results. Finally the two
 appendices \ref{stocmatrx} and \ref{cadpath} collect basic material on stochastic matrices and spaces of c\`{a}dl\`{a}g paths.
  In Appendix \ref{randomset} we give a broad picture of  the random frame  we work within.

 \smallskip


\section{\textbf{ Assumptions and preliminary results}}\label{settingpb}
\parskip +3pt
 In this section  we fix some notations and  write down the   problem with the standing assumptions.
 We also present some basic results on weakly coupled systems we will need in the following.

 \medskip

 We deal with a weakly coupled system of Hamilton-Jacobi equations of the form
\begin{equation}\label{e1} \tag{HJ$\alpha$}
H_{i}(x,Du_{i})+\sum_{j=1}^{m}a_{ij}u_{j}(x)=\alpha \quad \mbox{in } \mathbb{T}^{N}
\quad \mbox{ for every } i \in \{1,\cdots,m\},
\end{equation}

\noindent where $m \geq 2$,  $\alpha$ is a real constant, $A$ is an
$m \times m$ matrix, the so--called coupling  matrix,  and $H_{1},
\cdots,...,H_{m}$ are Hamiltonians. The $H_i$ satisfy the following
set of assumptions for all $i \in \{1,\cdots,m\}$ :
\begin{itemize}
\item[(H1)] $ H_{i}:\mathbb{T}^{N}\times\mathbb{R}^{N}\rightarrow\mathbb{R} \quad \mbox{is
continuous}$;
 \item[(H2)] $p\mapsto H_{i}(x,p) \quad \mbox{is convex for every } x \in
 \mathbb{T}^{N}
 $;
  \item[(H3)]$p\mapsto H_{i}(x,p) \quad \mbox{is superlinear for every } x \in \mathbb{T}^{N}. $
 \end{itemize}
The superlinearity condition (H3) allows to define the
corresponding Lagrangians through the Fenchel transform, namely
$$L_i(x,q)= \max_{p\in\mathbb{R}^{N}} \{p \cdot q - H_i(x,p)\} \quad\hbox{for any $i$}.$$
The coupling  matrix $A =(a_{ij})$  satisfies:
\begin{itemize}

\item[(A1)]$a_{ij}\leq 0 \mbox{ for every }i\neq j$;
\item[(A2)]$\sum \limits_{j=1}^{m}a_{ij} =0 \mbox { for any } i \in \{1,\cdots,m\};$
\item[(A3)] it is irreducible,
 i.e for every $W \subsetneq \{1,2,...,m\}$ there exists $i \in W $ and $j \notin W$ such that $a_{ij}<0.$
\end{itemize}
\smallskip
 Roughly speaking (A3) means that the system cannot split into independent subsytems. We remark that the assumptions (A1)
 and (A2) on the coupling matrix are equivalent to $e^{-At}$ being a stochastic matrix
for any $t\geq 0$ and due to irreducibility we get $e^{-A t}$ is
positive for any $t >0$, as made precise in Appendix
\ref{stocmatrx}.

We will consider (sub/super) solutions of the system in the
viscosity sense, see \cite{Davini}, \cite{siconolfi3} for the
definition. We recall that, as usual in convex coercive problems,
all the subsolutions are Lipschitz--continuous and the notions of
viscosity and a.e. subsolutions are equivalent.
\medskip

We now define the critical value $\beta$ as
$$\beta = \inf \{ \alpha \in \mathbb{R} \mid  \mbox{(HJ$\alpha$) admits subsolutions}\},$$
and write down  the critical system
\begin{equation}\label{ecritical} \tag{HJ$\beta$}
H_{i}(x,Du_{i})+\sum_{j=1}^{m}a_{ij}u_{j}(x)= \beta \quad \mbox{in }
\mathbb{T}^{N} \quad \mbox{ for every } i \in \{1,\cdots,m\}.
\end{equation}

\smallskip

The critical system (\ref{ecritical}) is the unique one in the
family \eqref{e1}, $\alpha \in \mathbb R$, for which there are
solutions. By critical (sub/super) solutions, we will mean
(sub/super) solutions of \eqref{ecritical}.

\smallskip
We deduce from the coercivity condition:

\smallskip
\begin{proposition} \label{lip} The family of subsolutions to \eqref{e1} are
equiLipschitz--continuous for any $\alpha \geq \beta$.

\end{proposition}

\smallskip

As  already recalled, a relevant property of systems is that not all
values in $\R^m$  are admissible for  subsolutions to (\ref{e1}) at
a given point of the torus. This rigidity phenomenon will play a
major role in what follows. We define for $\alpha \geq \beta$ and $x
\in \mathbb{T}^{N}$,

\begin{equation}\label{e42}
    F_\alpha(x)= \{ \mathbf{b} \in \mathbb {R}^{m} \mid \exists \;
    \mathbf{u} \; \mbox{subsolution to } \eqref{e1} \mbox{ with }\mathbf{u}(x)= \mathbf{b} \}.
\end{equation}

\smallskip

It is clear that
 $$\mathbf b \in F_\alpha (x) \; \Rightarrow \; \mathbf b + \lambda \,
 \mathbf 1 \in F_\alpha (x) \qquad\hbox{for any $\lambda \in \R$,}$$
where  $\mathbf{1}$ is the vector of $\mathbb R^m$ with all the
components equal to $1$. It is also apparent from the stability
properties of subsolutions and the convex character of the
Hamiltonians, that $F_\alpha$ is closed and convex at any $x$. We
proceed recalling the PDE definition of Aubry set $\mathcal A$ of
the system, see \cite{Davini}, \cite{siconolfi3}, \cite{SiZa}.
\begin{definition}\label{pdeau}
A point $y$ belongs to the Aubry set if and only if the maximal
critical subsolution  taking an admissible value at $y$ is a
solution to (\ref{ecritical}).
\end{definition}

\smallskip

\begin{definition}
A  critical subsolution $\mathbf u$    is said {\em locally strict}
at a point $y \in \mathbb T^N$ if there is a neighborhood $U$ of $y$
and a positive constant $\delta$ with
\[ H_i(x,Du_i) +\sum_{j=1}^{m}a_{ij}u_{j}(x) \leq \beta- \delta \quad\hbox{for any $i \in \{1, \cdots, m\}
$, a.e. $x \in U$.}\]
\end{definition}

We recall the  following property:
\smallskip

\begin{proposition} (\cite{Davini} Theorem 3.13) \label{DZ1}  A point
 $y \not \in \A$ if and only if  there exists a critical subsolution
 locally strict at $y$.
\end{proposition}

\medskip
As pointed out in the Introduction the admissible values make up a
one--dimensional set on $\A$.
\smallskip

\begin{proposition}\label{prop28} (\cite{Davini} Theorem 5.1) An element $y$  belongs  to the Aubry set
if and only if
$$F_\beta(y)= \{\mathbf{b}+\lambda \, \mathbf{1} \mid \lambda \in \R\}$$
where $\mathbf{b}$  is some vector in $\mathbb R^m$ depending on
$y$.
\end{proposition}

\smallskip

On the contrary, if $y \not\in\mathcal A$, the admissible set
possesses nonempty interior, which is characterized as follows:

\smallskip

\begin{proposition}\label{prop30}
Given $y \notin {\mathcal A} $,  the interior of $F_\beta(y)$ is
nonempty, and $\mathbf b \in \R^m$ is an internal point of
$F_\beta(y)$ if and only if there is a critical subsolution $\mathbf
u$ locally strict at $y$ with $\mathbf u(y)= \mathbf b$.
\end{proposition}
\begin{proof}
The values $\mathbf b $ corresponding to critical subsolutions
locally strict at $y$  make up a nonempty set in force of
Proposition \ref{DZ1}, it is in addition convex by the convex
character of the system. We will  denote it by $\widetilde
F_\beta(y)$.

Let $\mathbf b \in \widetilde F_\beta(y)$,  we claim  that  there
exists $\nu_0
>0$ with
\begin{equation}\label{notte1}
  \mathbf{b}+\nu \, \mathbf e_{i} \in \widetilde F_{\beta}(y)\, \qquad\mbox{ for  any } i, \, \nu_0 >  \nu >0.
\end{equation}
We denote by $\mathbf u$ the locally strict critical subsolution
with $\mathbf u(y)= \mathbf b$, then there exists $0<\eps < 1$ and
$\delta > 0$ such that
\begin{equation} \label{e51}
H_{i}(x,Du_{i}(x))+\sum_{j=1}^{m}a_{ij}u_{j}(x)\leq \beta - 2 \,
\delta \quad \mbox{ for any $i$, a.e. }x \in B(y,\eps).
\end{equation}
We  fix $i$ and  assume
\begin{equation} \label{e49}
\eta(\epsilon)+a_{ii}\,\frac {\epsilon^{2}}{2} < \delta
\qquad\hbox{for any $i$,}
\end{equation}
where $\eta$ is a continuity modulus for $(x,p) \mapsto H_i(x,p)$ in
$\mathbb T^N \times B(0,\ell_\beta+1)$  and $\ell_\beta$ is a
Lipschitz constant for all critical subsolutions, see Proposition
\ref{lip}.

We define $\mathbf w: \mathbb T^N \to \R^M$ via
\[w_{j}(x) = \left \{ \begin{array}{ll}
           u_{j}(x) & \mbox{if $ j\neq i$} \\
            \max \{ \phi(x),u_{i}(x)\} & \mbox{if $j=i$} \\
          \end{array} \right .\]
where
$$\phi(x):= u_{i}(x)- \displaystyle \frac{1}{2}\,|y-x|^{2}+
\displaystyle \frac{\epsilon^2}{2} $$
   Notice that
\begin{equation}\label{ecu}
 w_i= \phi > u_i \quad\hbox{in } \; B(y,\eps) \qquad\hbox{and } \quad
 w_i=u_i \quad\hbox{outside} \;
B(y,\eps).
\end{equation}

By \eqref{e51}, \eqref{e49} and the
 assumptions on the coupling matrix, we have for any $i$ and  a.e.  $x \in B(y,\eps)$

\begin{eqnarray*}
 && H_i(x,Dw_i(x)) + \sum_{j} a_{ij} \, w_j(x)  \\
&=& H_i(x,Du_i(x) +
  (y-x)) + \sum_{j \neq i} a_{ij} \, u_j(x) + a_{ii} \, \phi(x) \\
   &\leq& H_i(x,Du_i(x))+ \eta( \eps)  +  \sum_j a_{ij} \, u_j(x) + a_{ii} \, \frac{\eps^2 }2 \\
   &\leq& \beta -  2 \, \delta +  \delta  = \beta - \delta
\end{eqnarray*}
Further, for $j\neq i$ and for a.e. $x \in B(y,\epsilon)$, we have
\begin{eqnarray*}
   && H_{j}(x,Dw_{j}(x))+\sum_k a_{jk}w_{k}(x) \\
   &=& H_{j}(x,Du_{j}(x))+\sum_{k } a_{jk}u_{j}(x)+a_{ik}\left(-\displaystyle
   \frac{1}{2}\,|y-x|^{2}+ \displaystyle \frac{\epsilon^2}{2}\right) \\
   & \leq &  \beta - 2 \, \delta ,
\end{eqnarray*}
where the last inequality is due to the fact that $a_{ki} \leq 0$.
The previous computations  and \eqref{ecu} show that $\mathbf w$ is
a critical subsolution locally strict at $y$, and this property is
inherited by
\[\lambda \, \mathbf w + (1 - \lambda) \, \mathbf u\]
for any $\lambda \in [0,1]$. We therefore prove \eqref{notte1}
setting $\nu_0 = \frac {\eps^2}2$.

Taking into account that $\mathbf b + \lambda \, \mathbf 1 \in
\widetilde F_\beta(y)$ for any $\lambda \in \R$ and that the vectors
$\mathbf e_i$, $i=1, \cdots, m$, and  $- \mathbf 1$ are affinely
independent, we derive from  \eqref{notte1} and $\widetilde
F_\beta(y)$ being  convex, that $\mathbf b$ is an internal point of
$\widetilde F_\beta(y)$ and consequently that  $\widetilde
F_\beta(y)$  is an open set. Finally it is also dense in
$F_\beta(y)$ because if $\mathbf v$ is any critical subsolution and
$\mathbf u$ is in addition locally strict at $y$ then any convex
combination of $\mathbf u$ and $\mathbf v$ is locally strict and
\[ \lambda \, \mathbf u(y) + (1 - \lambda) \, \mathbf v(y) \to  \mathbf
v (y) \qquad\hbox{as $\lambda \to 0$.}\] The property of being open,
convex and dense in $F_\beta(y)$ implies that $\widetilde
F_\beta(y)$ must coincide with the interior of $ F_\beta(y)$, as
claimed.

$\hfill{\Box}$
\end{proof}
 \medskip

\medskip

We  proceed introducing the action functional defined in
\cite{siconolfi3} on which  our analysis is based, see Appendices
\ref{cadpath}, \ref{randomset} for terminology, notation,
definitions and basic facts.

 \medskip

Given $\alpha \geq \beta$ and an initial point $x \in
\mathbb{T}^{N}$, the  action functional adapted to the system is
$$\mathbb E_{\mathbf a} \left [ \int_0^\tau
 L_{\omega(s)}(x
+ \mathcal {I}(\Xi)(s),-\Xi(s)) + \alpha \, ds \right ],$$
where $\mathbf a$ is any probability vector of $\R^m$, $\tau$ a bounded stopping time and $\Xi$ a control.

 \medskip
 Using the action functional, we get the following characterizations of subsolutions to the
 system and admissible values:
\begin{theorem}\label{t17}
A function $\mathbf u:\mathbb{T}^{N} \rightarrow \mathbb{R}^{m}$ is
a subsolution of (\ref{e1}), for any $\alpha \geq \beta$, if and
only if
$$\mathbb{E}_{\mathbf {a}} \big [u_{\omega(0)}(x) - u_{\omega(\tau)}(y) \big ]
   \leq  \mathbb{E}_{\mathbf {a}}  \left [
\int_0^\tau L_{\omega(s)}(x + \mathcal {I}(\Xi)(s),-\Xi(s)) + \alpha
\, ds \right],$$ for any pair of points $x$, $y$  in
$\mathbb{T}^{N}$, probability vector  $\mathbf {a}  \in \R^m$, any
bounded
stopping time $\tau$ and $\Xi \in \mathcal{K}(\tau, y-x)$.\\

\end{theorem}

\medskip

\begin{theorem}\label{t18}
Given $y \in \mathbb{T}^{N}$,  $\alpha \geq \beta$,  $\mathbf{b} \in
\mathcal{F}_{\alpha}(y) $ if and only if
\begin{equation}\label{e41}
 \mathbb{E}_{i} \left  [\int_0^\tau L_{\omega(s)}(y + \mathcal I(\Xi)(s),-\Xi(s)) + \alpha \, ds - b_{i}
+ b_{\omega(\tau)} \right ] \geq 0,
\end{equation}
for any $i \in \{1, \cdots, m\}$, bounded  stopping times $\tau$ and $\tau$--cycles $\Xi$.
\end{theorem}
\section{Properties of stopping times}\label{propstoptim}

Given a stopping time $\tau$,  the push--forward  of
$\mathbb{P}_{\mathbf{a}}$ through $\omega(\tau)$ is a probability
measure on indices $\{1,\cdots,m\}$, which can be identified with an
element of the simplex (denoted by $\mathcal S$) of probability
vectors in $\R^m$. Then
$$\mathbf{a} \mapsto \omega(\tau) \# \mathbb{P}_{\mathbf{a}},$$
defines a map from  $\mathcal S$ to $\mathcal S$ which is, in
addition, linear. Hence, thanks to Proposition \ref{prop32}, it can
be represented by a stochastic matrix, which we denote by $e^{- A
\tau}$,  acting on the right, i.e.
\begin{equation} \label{e58}
\mathbf{a} \,  e^{- A \tau }  = \omega(\tau) \#
\mathbb{P}_\mathbf{a} \qquad \mbox{for any } \mathbf{a} \in
\mathcal{S}.
\end{equation}

\medskip
\begin{definition} \label{charactsc} We say that $\mathbf{a} \in
\mathcal S$ is a {\em characteristic vector} of $\tau$ if it is an
eigenvector of $e^{- A \tau}$  corresponding to the eigenvalue $1$, namely $\mathbf {a} = \mathbf {a} \, e^{- A\tau}$.
\end{definition}
\smallskip

\begin{rem}\label{now} According to Proposition \ref{prop33},  any stopping time possesses a
characteristic vector $\mathbf{a}$, and
\[ \mathbb E_\mathbf{a} b_{\omega(\tau)} = \mathbf{a} \, e^{- A \tau} \cdot \mathbf{b}= \mathbf{a} \cdot \mathbf{b} \qquad\hbox{for every $\mathbf{b} \in
\mathbb R^m$.}\]
\end{rem}

\smallskip
According to the remark above,  Theorem \ref{t17} takes a simpler
form if we just consider expectation operators $\mathbb E_{\mathbf
a}$ with $\mathbf a$ characteristic vector. This result will play a
key role in Lemma \ref{post21}.

\smallskip

\begin{corollary} \label{cornow}
A function $\mathbf u$ is a subsolution to \eqref{e1} if and only if
\begin{equation}\label{charstric}
 \mathbf a \cdot \big ( \mathbf u(x)- \mathbf u(y) \big )
\leq \mathbb E_{\mathbf a} \left ( \int_0^\tau L_{\omega(s)}(x +
\mathcal{I}(\Xi)(s),-\Xi(s)) + \beta  \, ds \right ),
\end{equation}
 for any $i \in
\{1, \cdots, m\}$, bounded  stopping times $\tau$, $\mathbf {a}$
characteristic vector of $\tau$,  and  $\Xi \in \mathcal K(\tau,
y-x)$.
\end{corollary}
\smallskip

\begin{lemma}\label{media}
Take  $\tau_n$ as in \eqref{splstopping}.  Then
\[ e^{-A \tau_n} \to e^{-A \tau} \qquad\hbox{as $n$ goes to
infinity.}\]
\end{lemma}
\begin{proof}
Let $\mathbf{a} \in \mathcal S$, $\mathbf{b} \in \mathbb{R}^{m}$.
Being $\omega$ right-continuous and $\tau_n \geq \tau$, we get
$\omega(\tau_n) \to \omega(\tau)$ for any $\omega \in \mathcal D$,
and consequently
\[ b_{\omega(\tau_n)} \to b_{\omega(\tau)}.\]
This implies, taking into account \eqref{e58}
\[ (\mathbf{a} \, e^{- A \tau_n}) \cdot \mathbf{b} =  \mathbb E_\mathbf{a} b_{\omega(\tau_n)} \to \mathbb E_\mathbf{a}
b_{\omega(\tau)}= (\mathbf{a} \, e^{- A \tau}) \cdot \mathbf{b},\]
and yields the assertion.
 $\hfill{\Box}$
\end{proof}
\begin{definition}\label{superstrict} Given any positive constant $\epsilon$, we say that $\tau$ is strongly greater
than $\epsilon$, written mathematically as  $\tau \gg \epsilon$, to
mean that $\tau - \epsilon$ is still a stopping time, or
equivalently
\begin{equation}\label{greatergreater}
\tau \geq \epsilon \;\hbox{a.s.} \quad\hbox{and} \quad  \{ \tau \leq t\}
\in \mathcal{F}_{t- \epsilon } \quad\hbox{for any $t \geq \epsilon$.}
\end{equation}
Moreover for $i \in \{1, \cdots, m\}$, we say
$$\tau \gg \epsilon  \;\hbox{in $\mathcal D_i$}$$ to mean
\begin{equation}\label{greatergreater1}
\tau \geq \epsilon \;\hbox{a.s. in $\mathcal D_i$} \quad\hbox{and} \quad  \{
\tau \leq t\} \cap \mathcal D_i \in \mathcal F_{t- \epsilon } \quad\hbox{for any $t \geq
\epsilon$.}
\end{equation}
\end{definition}
\medskip

\begin{proposition}\label{corollozzo} Let $\epsilon >0$, $i \in \{1, \cdots, m\} $.
Then for every  $\tau \gg \epsilon$ in $\mathcal D_i$, there exists a positive constant
$\rho$, solely depending on $\epsilon$ and on the coupling
matrix, such that
\begin{equation}\label{lozzo1}
  \left ( e^{- A \tau} \right )_{ij}  > \rho \qquad j \in \{1, \cdots,
  m\}.
\end{equation}
\end{proposition}
\begin{proof} We  approximate $\tau$ by a sequence of simple stopping
times $\tau_n$ with $\tau_n   \geq \tau$, as indicated in
Proposition \ref{stoppapproxi}. For a fixed  $n$, we then have
$$\tau_n=  \sum_j \frac j{2^n} \, \mathbb{I}(\{\tau \in [(j-1)/2^n,j/2^n)\}).$$
By  the assumption on $\tau$, the set $F_j := \{\tau \in
[(j-1)/2^n,j/2^n)\} \cap \mathcal D_i$ belongs to $\mathcal F_{
\frac j{2^n}  - \epsilon}$.  By applying  Lemma \ref{lem4}, we
therefore get
\begin{eqnarray*}
  \mathbf{e}_i \, e^{-A \tau_n} &=& \omega(\tau_n) \# \mathbb{P}_i=
\sum_j \, \omega( j/2^n)\# (\mathbb{P}_i \mres F_j) = \left ( \sum_j
\omega(j/2^n - \epsilon )\# (\mathbb{P}_i \mres F_j) \right ) \,
e^{-A \epsilon}
  \\ &=& \big ( \omega (\tau_n - \epsilon )\# \mathbb{P}_i \big )\, e^{-A \epsilon}
\end{eqnarray*}
Owing to $\omega (\tau_n - \epsilon )\#\mathbb{P}_i \in \mathcal S$, we deduce
\[ \mathbf{e}_i \,  e^{- A \tau_n} \in \{ \mathbf{b} \, e^{-A \epsilon} \mid \mathbf{b} \in \mathcal S\}, \]
 we have in addition $   e^{- A \tau_n} \to e^{- A \tau}$  by Lemma \ref{media},
and consequently
\[ \mathbf{e}_i \,  e^{- A \tau} \in \{ \mathbf{b} \, e^{-A \epsilon} \mid \mathbf{b} \in \mathcal S\}. \]
This set is compact, and contained in the relative interior of
$\mathcal S$ because  $e^{-A \epsilon}$ is positive by Proposition
\ref{prop35}. Since the components of $\mathbf{e}_i \,  e^{- A
\tau}$ make up the $i$--th  row of $e^{- A \tau}$, we immediately
derive the assertion.
 $\hfill{\Box}$
\end{proof}
\smallskip

According to the previous proposition and Proposition \ref{prop34},
the the characteristic vector of a $\tau \gg \epsilon$, for some
$\epsilon >0$,   is unique and positive.

\smallskip

 \begin{rem}\label{corollozzobis} Take  $\tau \gg \eps$ and denote by $\rho$
-the positive constant satisfying \eqref{lozzo1} for any $i$, $j$,
according to Proposition \ref{corollozzo}. Then, since $e^{- A
\tau}$ is a stochastic matrix, we have
\[ \left ( e^{- A \tau} \right )_{ij}  = 1 - \sum_{k \neq j}
\left ( e^{- A \tau} \right )_{ik} \leq 1 - (m-1) \rho \leq 1 -
\rho.\]
\end{rem}

\smallskip

\begin{rem}\label{min}  Let  $\tau $,  $\rho$
be as in the previous remark. If $\mathbf a$ is the characteristic
vector of $\tau$ then we get for any $i$
\[a_i = \sum_j a_j \, \left (e ^{- A \tau} \right )_{ji} > \rho.\]
\end{rem}

\medskip

 For any stopping time $\tau$, we consider the shift flow $\phi_\tau$ on $\mathcal D$ defined by :
\begin{eqnarray*}
  \phi_\tau: \mathcal D &\to& \mathcal D \\
  \omega &\mapsto& \omega(\cdot + \tau(\omega)).
  \end{eqnarray*}
 We proceed by establishing some related properties.

\smallskip

\begin{lemma}\label{convergeflow}
Assume that $\tau_n$ is a sequence of stopping times converging to
$\tau$ uniformly in $\mathcal D$, then
\[ \phi_{\tau_n} \to \phi_{\tau} \qquad\hbox{as $n \to + \infty$}\]
pointwise in $\mathcal D$, with respect to the Skorohod convergence,
see Appendix \ref{cadpath} for the definition.

\end{lemma}
\begin{proof}
We fix $\omega \in \mathcal D$,  we  set
\[g_n(t)= t + \tau(\omega) - \tau_n(\omega) \qquad\hbox{for any $n$, $t \geq 0$}.\]
We have for any $t$
\[\phi_\tau(\omega)(t)= \omega(t+ \tau_n(\omega) + (\tau(\omega) -
\tau_n(\omega)))= \phi_{\tau_n}(\omega)(g_n(t) ).\] This yields the
asserted convergence because the $g_n$ are a  sequence of strictly
increasing functions uniformly converging to the identity.
$\hfill{\Box}$
\end{proof}
\medskip

\begin{proposition}\label{flowcont} The shift flow $\phi_\tau: \mathcal D \to \mathcal D$ is measurable.
\end{proposition}
\begin{proof} If $\tau$ is a simple stopping
time, say of the form $\tau= \sum_k t_k \, \mathbb I(E_k)$, then
\[\phi_\tau(\omega)= \sum_k \phi_{t_k}(\omega) \,  \mathbb I(E_k)(\omega)\]
and the assertion follows being $\phi_{t_k}$ measurable for any $k$,
$\mathbb I(E_k)$ measurable. If $\tau$ is not simple then, by Proposition \ref{splstopping}, there exists
a sequence of simple stopping times $\tau_n$ converging to $\tau$
uniformly in $\mathcal D$, this  implies that $\phi_\tau$ is measurable as
well, as pointwise limit of measurable maps, in force of Lemma
\ref{convergeflow}.
 $\hfill{\Box}$
\end{proof}
\medskip

We now define the probability measure $\phi_\tau \# \mathbb
P_{\mathbf{a}}$, for $\mathbf{a} \in \mathcal S$.  The following
result generalizes Proposition \ref{pushprbdet} to shifts given for
stopping times. It will be used in Theorem \ref{cod}  and in Lemma
\ref{preprese}.
\smallskip

\begin{theorem}\label{flowflow} Let $\mathbf{a}$ be a probability vector, then
$$ \phi_\tau \# \mathbb P_{\mathbf{a}} = \mathbb P_{\mathbf{a}\,e^{-A \tau}} .$$
\end{theorem}
\medskip

We need the following preliminary result:

\smallskip

\begin{lemma}\label{preflowflow} Let $\mathbf{a}$, $t$, $E$ be a vector in $\mathcal S$,
  a positive deterministic time and a set in
$\mathcal F_t$, respectively, then
\[ \phi_t \# (\mathbb P_\mathbf{a} \mres E) = \mathbb P_{\mathbf{a}}(E) \, \mathbb P_{\mathbf{b}}\qquad\hbox{for some  $\mathbf{b} \in \mathcal S $.} \]
\end{lemma}
\begin{proof}
We first assume $E$ to be a cylinder, namely
\[E= \mathcal C(t_1, \cdots,t_k; j_1, \cdots, j_k)\]
for some times and indices, notice that the condition $E \in \mathcal F_t$
implies $t_k \leq t$.  We fix $i \in \{1, \cdots, m\}$ and consider a cylinder
$C \subset \mathcal D_i$, namely
\[C= \mathcal C(0, s_2, \cdots, s_m; i, i_2,  \cdots, i_m)\]
for some choice of times and indices. We  set
\[F= \{\omega \mid \phi_t(\omega) \in C\} \cap E,\]
then
\[F=\mathcal C(t_1, \cdots,t_k, t, t+s_2, \cdots,t+s_m; j_1, \cdots, j_k, i,
i_2, \cdots, i_m).\] We have
\begin{eqnarray*}
   \phi_t \# (\mathbb P_\mathbf{a} \mres E)(C) &=& \mathbb P_\mathbf{a} (F) \\ &=& \left (\mathbf{a} \,e^{-A t_1} \right )_{j_1} \,
\prod_{l=2}^{k} \left (e^{-A (t_l-t_{l-1})} \right )_{j_{l-1} \,
j_l} \, \left (e^{-A (t- t_k)} \right )_{ j_k \, i} \,
\prod_{r=2}^{m} \left (e^{-A (s_r-s_{r-1})} \right )_{i_{r-1} \, i_r} \\
 &=& \mathbb P_\mathbf{a} (E) \, \left (e^{-A (t- t_k)} \right )_{ j_k \, i} \,
\prod_{r=2}^{m} \left (e^{-A (s_r-s_{r-1})} \right )_{i_{r-1} \,
i_r},
\end{eqnarray*}
we also have
\[ \mathbb P_i(C)= \prod_{r=2}^{m} \left (e^{-A (s_r-s_{r-1})} \right )_{i_{r-1} \,
i_r},\] and we consequently get the relation
\[ \phi_t \# ( \mathbb P_\mathbf{a} \mres E)(C)= \mathbb P_\mathbf{a} (E) \, \mu_i \, \mathbb P_i(C)\]
with
\begin{equation}\label{floflo1}
    \mu_i=  \left (e^{-A (t- t_k)} \right )_{ j_k \,
i}
\end{equation}
 just depending on $E$ and $i$. If $C$ is any cylinder, we write
 \begin{equation}\label{floflo2}
   \phi_t \# (\mathbb P_\mathbf{a} \mres E)(C) = \sum_i  \phi_t \# (\mathbb P_\mathbf{a}\mres
E)(C \cap \mathcal D_i) =  \mathbb P_\mathbf{a}(E) \, \sum_i  \mu_i
\, \mathbb P_i(C)
\end{equation}
 where
the $\mu_i$ are defined as in \eqref{floflo1}. Taking into account
that $\mu_i \geq 0$ for any $i$ and $\sum_i \mu_i =1$, $\mathbf{b}
:= \sum_i \mu_i \, \mathbf{e}_i \in \mathcal S$,  we derive from
\eqref{floflo2}
\[\phi_t \# (\mathbb P_\mathbf{a} \mres E)(C) =  \mathbb P_\mathbf{a}(E) \, \mathbb P_\mathbf{b}(C).\]
This in turn implies by Proposition \ref{klenke}
\begin{equation}\label{floflo3}
   \phi_t \# (\mathbb P_\mathbf{a}\mres E) =  \mathbb P_\mathbf{a}(E) \, \mathbb P_\mathbf{b}
\end{equation}
showing the assertion in the case where $E$ is a cylinder.
If instead $E$ is a multi--cylinder, namely $E=\cup_j E_j$ with
$E_j$ mutually disjoint cylinders then by the previous step
\[\phi_t \# (\mathbb P_\mathbf{a}\mres E) = \sum_j \phi_t \#(\mathbb P_\mathbf{a} \mres E_j) = \sum_j \mathbb P_\mathbf{a}(E_j) \, \mathbb P_{\mathbf{b}_j}\]
which again implies \eqref{floflo3} with \[ \mathbf{b} = \sum_j
\frac{\mathbb P_\mathbf{a}(E_j)}{\mathbb P_\mathbf{a}(E)}  \,
\mathbf{b}_j.\] Finally, for a general $E$, we know from Proposition
\ref{klenke} that there is a sequence of multi--cylinders $E_n$ with
\begin{equation}\label{flowflow2}
   \lim_n \mathbb P_\a(E_n \triangle E)=0.
\end{equation}
Given $F \in \F$, we set
\[C = \{\omega \mid \phi_t(\omega) \in F\}, \]
 we have
\[ \phi_t \# (\mathbb P_\a \mres E_n)(F)= \mathbb P_\a  (C  \cap E_n  ) \leq \mathbb P_\a \big
((C\cap E) \cup (E \triangle E_n) \big ) =  \phi_t \# (\mathbb P_\a
\mres E)(F) + \mathbb P_\a (E \triangle E_n) \] and similarly
\[ \phi_t \# (\mathbb P_\a \mres E)(F) \leq  \phi_t \# (\mathbb P_\a
\mres E_n)(F) + \mathbb P_\a (E \triangle E_n) .\] We deduce  in
force of  \eqref{flowflow2}
\[\lim_n \phi_t \# (\mathbb P_\a \mres E_n)(F) = \phi_t \# (\mathbb P_\a \mres
E)(F)\] which in turn implies that $\phi_t \# (\mathbb P_\a \mres
E_n)$ weakly converges to $\phi_t \# (\mathbb P_\a \mres E)$. Since,
by the previous step in the proof
\[\phi_t \# (\mathbb P_\a \mres E_n)=
\mathbb P_\a(E_n) \, \mathbb P_{\b_n} \qquad\hbox{ for some $\b_n
\in \mathcal S$}\] we derive from Proposition \ref{conti} and
\eqref{flowflow2}
\[\phi_t \# (\mathbb
P_\a \mres E)= \mathbb P_\a(E) \, \mathbb P_{\b} \qquad\hbox{with
$\b =\lim_n \b_n$. }\] This concludes the proof.

 $\hfill{\Box}$
\end{proof}
\\
\begin{proof} ({\em of the Theorem \ref{flowflow}}) \\
 We first show that
 \begin{equation}\label{flowflow1}
   \phi_\tau \# \mathbb P_\mathbf{a} = \mathbb P_\mathbf{b}\qquad\hbox{for a suitable $\mathbf{b} \in
   \mathcal S$.}
\end{equation}
If $\tau = \sum_k t_k \, \mathbb I(E_k)$ is simple then by Lemma
\ref{preflowflow}
\[\phi_\tau \# \mathbb P_\mathbf{a} =  \sum_k \phi_{t_k} \# (\mathbb P_\mathbf{a} \mres E_k) = \sum_k \mathbb P_\mathbf{a}(E_k) \, \mathbb P_{\mathbf{b}_k}\]
for some $b_k \in \mathcal S$, and we deduce  \eqref{flowflow1} with
$\mathbf{b}= \sum_k \mathbb P_\mathbf{a}(E_k) \, \mathbf{b}_k$.

Given a general stopping time $\tau$, we approximate it by a
sequence of simple stopping times $\tau_n$, and, exploiting the
previous step, we consider $\mathbf{b}_n \in \mathcal S$ with
\[\phi_{\tau_n} \# \mathbb P_\mathbf{a}= \mathbb P_{\mathbf{b}_n}.\]
We know from Lemma \ref{convergeflow}  that
\[ \phi_{\tau_n}(\omega) \to \phi_\tau(\omega) \qquad\hbox{for any $\omega$
in the Skorohod sense,}\] and we  derive via  Dominate Convergence
Theorem
\[ \mathbb E_\mathbf{a} f(\phi_{\tau_n}) \to \mathbb E_\mathbf{a} f( \phi_\tau)\]
for any bounded measurable function $f: \mathcal D \to \mathbb R$. Using change of
variable formula (\ref{changevar}) we get
\[ \int_\mathcal D f \, d \phi_{\tau_n} \#  \mathbb P_\mathbf{a} \to \int_\mathcal D f \, d \phi_{\tau} \#
\ \mathbb P_\mathbf{a}\] or equivalently
\[ \mathbb P_{\mathbf{b}_n} = \phi_{\tau_n} \#  \mathbb P_\mathbf{a} \to \phi_{\tau} \#  \mathbb P_\mathbf{a}\]
 in the sense of weak convergence of measures. This in turn implies  by the  continuity
property stated in Proposition \ref{conti} that $\mathbf{b}_n$ is
convergent in $\mathbb R^{m}$ and
\[\mathbb P_{\mathbf{b}_n} \to \mathbb P_{\mathbf{b}}\quad\qquad\hbox{with $\mathbf{b}= \lim_n \mathbf{b}_n$}\]
which   shows \eqref{flowflow1}.  We can compute the components of
$\mathbf{b}$  via
\[ b_i= \mathbb P_\mathbf{a}\{\omega\mid \phi_\tau(\omega) \in \mathcal D_i\}= \mathbb P_\mathbf{a}\{\omega \mid \omega(\tau(\omega))
=i\}= \big (\omega(\tau) \# \mathbb P_\mathbf{a} \big )_i= \big(
\mathbf{a} \, e^{- A \tau} \big )_i.\] This concludes the proof.
$\hfill{\Box}$
\end{proof}
\section{Cycle iteration}\label{reptcyc}
\parskip +3pt

It is immediate that we can  construct  a sequence of
(deterministic) cycles going through a given closed  curve any
number of times. We aim at generalizing this iterative procedure in
the random setting we are working with, starting from
$\tau^0$--cycle, for some stopping time $\tau^0$.  In this case the
construction is more involved and requires some details.

\smallskip

Let $\tau^0$, $\Xi^0$  be a  simple stopping time and a
$\tau^0$--cycle, respectively, we recursively define
 for $j \geq 0$
 \begin{equation}\label{buckle1}
    \tau^{j+1}(\omega) = \tau^0(\omega) + \tau^j(\phi_{\tau^0}(\omega))
\end{equation}
and
\begin{equation}\label{cape1}
   \Xi^{j+1}(\omega)(s) = \left \{
\begin{array}{ll}
    \Xi^j(\omega)(s), & \hbox{for  $s \in [0,\tau^j(\omega))$} \\
    \Xi^0(\phi_{\tau^j}(\omega))(s-\tau^j(\omega)) & \hbox{for  $s \in [\tau^j(\omega),+ \infty)$.}
\end{array}
\right.
\end{equation}

\smallskip

We will prove below that the $\Xi^j$ make up the sequence of
iterated random cycles we are looking for. A first step is:

\medskip

\begin{proposition}\label{cedro} The
 $\tau^j$, defined by \eqref{buckle1},  are  simple stopping times for
 all $j$. If, in addition, $\tau^0 \gg \delta$ in $\mathcal D_i$, for some $i \in \{1, \cdots, m\}$, $\delta >0$,
 then $\tau^j \gg \delta$ in $\mathcal D_i$.
\end{proposition}

\begin{proof} We argue by induction on $j$. The property is true for
$j=0$, assume by inductive step that $\tau^j$ is a simple stopping
time, then by Proposition \ref{flowcont} $\tau^{j+1}$ is a random
variable, as sum and composition of measurable maps,  taking
nonnegative values. Assume
\begin{eqnarray}
  \tau^0 &=& \sum_{l=1}^{m_0} s_l \, \mathbb I(F_l)  \label{cedro1}\\
  \tau^j &=& \sum_{k=1}^{m_j} t_k \, \mathbb I(E_k) \label{cedro2}
\end{eqnarray}
then the sets
\[F_l \cap \{\omega \mid \phi_{\tau^0}(\omega) \in E_k\} \qquad l=1, \cdots,
m_0, \; k=1, \cdots, m_j\] are mutually disjoint and their union is
the whole $\mathcal D$. Moreover if
\[\omega \in F_l \cap \{\omega \mid \phi_{\tau^0}(\omega) \in E_k\}\]
then
\[\tau^{j+1}(\omega) = \tau^0 (\omega) + \tau^j(\phi_{\tau^0}(\omega))= s_l
+t_k,\] which shows that $\tau^{j+1}$ is simple. Since $\tau^0$,
$\tau^j$ are stopping time then $F_l \in \mathcal F_{s_l}$ and $E_k \in
\mathcal F_{t_k}$. By Proposition \ref{supershift}
\[F_l \cap \{\omega \mid \phi_{\tau^0}(\omega) \in E_k\} \in \mathcal F_{s_l
+t_k},\] which shows that $ \tau^{j+1}$ is a stopping time.\\
Moreover if $\tau^0 \gg \delta$ in $\mathcal D_i$ then $F_l \cap
\mathcal D_i  \in \mathcal F_{s_l-\delta}$ and consequently
\[F_l \cap  \mathcal D_i \cap \{\omega \mid \phi_{\tau^0}(\omega) \in E_k\} \in \mathcal F_{s_l
+t_k -\delta},\] which shows that $ \tau^{j+1} \gg \delta$ in
$\mathcal D_i$.

$\hfill{\Box}$
\end{proof}
\bigskip

\medskip

The main result of the section is

\smallskip

\begin{theorem}\label{cod} The $\Xi^j$, as defined in \eqref{cape1},
 are  $\tau^j$--cycles for all $j$.
\end{theorem}

\smallskip

A lemma is preliminary.

\smallskip

\begin{lemma}\label{precod} For any $j$, $\omega$
\[\tau^{j+1}(\omega)= \tau^j(\omega) + \tau^0(\phi_{\tau^j}(\omega)).\]
\end{lemma}
\begin{proof} Given  $j \geq 1$, we preliminarily  write
\begin{eqnarray*}
  \phi_{\tau^{j-1}} ( \phi_{\tau^0}(\omega))(s) &=& \phi_{\tau^0}(\omega)(s + \tau^{j-1}(\phi_{\tau^0}(\omega) ) \\
   &=& \omega(s + \tau^0(\omega) + \tau^{j-1}(\phi_{\tau^0}(\omega) ) = \omega(s +
\tau^j(\omega))= \phi_{\tau^j}(\omega)(s)
\end{eqnarray*}
which gives
\begin{equation}\label{precod1}
   \phi_{\tau^{j-1}} \circ \phi_{\tau^0}= \phi_{\tau^j}.
\end{equation}
We proceed arguing by induction on $j$. The formula in the statement
is true for $j= 0$. We proceed showing that
 it is true  for $j+1$ provided it holds for $j \geq 0$.  We
 have, taking into account \eqref{precod1}
 \begin{eqnarray*}
   \tau^{j+1}(\omega) &=& \tau^0(\omega) + \tau^j(\phi_{\tau^0}(\omega)) =
\tau^0(\omega) + \tau^{j-1}(\phi_{\tau^0}(\omega))  +
\tau^0(\phi_{\tau^{j-1}}(\phi_{\tau^0}(\omega))) \\
    &=& \tau^j(\omega) +
\tau^0(\phi_{\tau^{j-1}}(\phi_{\tau^0}(\omega))) = \tau^j(\omega) +
\tau^0(\phi_{\tau^j}(\omega))
 \end{eqnarray*}
 as asserted.
 $\hfill{\Box}$
\end{proof}

\medskip

\begin{proof} (of Theorem \ref{cod})
The property is true for $j=0$, then we argue by induction on $j$.
We exploit the principle that $\Xi^j$ is a control  if and only the
maps $\omega \mapsto \Xi^j(\omega)(s)$ from $\mathcal D $ to
$\mathbb R^N$ are $\mathcal F_s$--measurable for all $s$. Given $s$
and a  Borel set $B$ of $\mathbb R^N$, we therefore need to show
\begin{equation}\label{cedro00}
\{\omega \mid \Xi^{j+1}(\omega)(s) \in B\} \in \mathcal F_s,
\end{equation}
 knowing  that $\Xi^0$, $\Xi^j$ are controls, the first by assumption and the latter
by inductive step. We set
\[E = \{ \tau^j > s\},\]
then we have by the very definition of $\Xi^{j+1}$\begin{equation}\label{cedro0}
   \{\omega \mid \Xi^{j+1}(s) \in B\}= F_1 \bigcup  F_2
\end{equation}
 with
\begin{eqnarray*}
  F_1 &=& \{\omega \mid \Xi^j(s) \in B\} \cap E  \\
 F_2 &=& \{\omega \mid
\Xi^0(\phi_{\tau^j}(\omega))(s-\tau^j(\omega)) \in B \} \setminus E.
\end{eqnarray*}
We know that
\begin{equation}\label{cedro10}
   F_1 \in \mathcal F_s,
\end{equation}
because $\tau^j$ is a stopping time and $\Xi^j$ a control. Assume
now $\tau^j$ to be of the form
\[\sum_{k=1}^{m_j} t_k \, \mathbb I(E_k)\]
then $E_k \setminus E = E_k$ or $E_k \setminus E = \emptyset$
according on whether $t_k \leq s$ or $t_k > s$ and so
\[ F_2 =  \bigcup_{t_k \leq s} \, \{\omega \mid \Xi^0(\phi_{t_k}(\omega))(s-t_k) \in B
\} \cap E_k \] Consequently, if $t_k \leq s$,  $\Xi^{j+1}(s)$ is
represented in $E_k $ by the composition of the following maps
\[ \omega  \xrightarrow{\psi_1} \phi_{t_k}(\omega) \xrightarrow{\psi_2} \Xi^0(\phi_{t_k}(\omega))\xrightarrow{\psi_3}
\Xi^0(\phi_{t_k}(\omega))(s-t_k).\] By the very definition of the
$\sigma$--algebra $\mathcal F'_t$, $\psi_3^{-1}(\mathcal B)  \subset
\mathcal F'_{s-t_k}$, moreover,  since $\Xi^0$  is adapted then
$\psi_2^{-1}(\mathcal F'_{s-t_k}) \subset \mathcal F_{s-t_k}$ and
finally $\psi_1^{-1}(\mathcal F_{s-t_k}) \subset \mathcal F_{s}$ by
Proposition \ref{supershift}. We deduce, taking also into account
that $E_k \in \mathcal F_{t_k} \subset \mathcal F_{s}$, that if $t_k
\leq s$ then
\[\{\omega \mid \Xi^{j+1}(s) \in B\} \cap E_k = \{\omega \mid \Xi^0(\phi_{t_k}(\omega))(s-t_k) \in B
\} \cap E_k \in \F_s,\] and consequently $F_2$, being the union of
sets in $\F_s$, belongs to $\mathcal F_s$ as well. By combining this
information with \eqref{cedro0}, \eqref{cedro10}, we prove
\eqref{cedro00} and conclude that $\Xi^{j+1}$ is a control.\\
To show that $ \Xi^{j+1}$ is a $ \tau^{j+1}$--cycle, we use the very
definition of $\tau^{j+1}$, $\Xi^{j+1}$ and write for any $\omega$
\begin{equation}\label{cod1}
    \int_0^{ \tau^{j+1}(\omega)}  \Xi^{j+1}(\omega) \, ds = I(\omega) +
    J(\omega)
\end{equation}
with
\begin{eqnarray*}
  I(\omega) &=& \int_0^{ \tau^j(\omega)}  \Xi^j(\omega) \,
ds \\
  J(\omega) &=& \int_{\tau^j(\omega)}^{ \tau^{j+1}(\omega)}
\Xi^0(\phi_{\tau^j}(\omega))(s-\tau^j(\omega)) \, ds.
\end{eqnarray*}
Due to $\Xi^j$ being a $\tau^j$--cycle, we have
 \begin{equation}\label{cod2}
     I(\omega) = 0 \quad\hbox{a.s.}
\end{equation}
We change the variable  in $J$, setting $t= s - \tau^j(\omega)$, and
exploit Lemma \ref{precod} to get
\begin{equation}\label{cod20}
    J(\omega)= \int_0^{\tau^0(\phi_{\tau^j}(\omega))}
\Xi^0(\phi_{\tau^j}(\omega))(t) \, dt.
\end{equation}
Let $E$ be any set in $\mathcal F$ and $\mathbf{a}$ a positive
probability vector. We integrate $J(\omega)$  over $E$ with respect
to $\mathbb P_\mathbf{a}$ using \eqref{cod20}, replace
$\phi_{\tau^j}(\omega))$ by $\omega$ via  change of variable
formula, and exploit Theorem \ref{flowflow}. We obtain
\begin{equation}\label{cod30}
    \int_E J(\omega) \, d\mathbb P_\mathbf{a} = \int_{\phi_{\tau^j}(E)} \left (\int_0^{\tau^0(\omega))} \Xi^0(\omega)(t) \,
dt \right ) \, d \mathbb P_{\mathbf{a} e^{- A \tau^j}} .
\end{equation}
Due to $\Xi_0$ being a $\tau^0$--cycle
\[ \int_0^{\tau^0(\omega)} \Xi^0(\omega)(t) \,
dt = 0 \qquad\hbox{a.s,}\] and therefore the integral in the right
hand--side of \eqref{cod30} is vanishing and so
\[ \int_E J(\omega) \, d\mathbb P_\mathbf{a}  =0.\]
Since $E$ has been arbitrarily chosen in $\mathcal F$ and
$\mathbf{a} >0$, we deduce in force of Lemma \ref{nullo}
\[J(\omega)= 0 \qquad\hbox{a.s.}\]
This information combined with \eqref{cod1}, \eqref{cod2} shows that
$\Xi^{j+1}$ is a $\tau^{j+1}$--cycle and conclude the proof.

$\hfill{\Box}$
\end{proof}

\smallskip

\section{Dynamical characterization of the Aubry set }\label{dynpropaub}

In this section we give  the main results of the paper on the cycle
characterization of the Aubry set.
\smallskip
As explained in the Introduction, a key step is to establish a
strengthened version of Theorem \ref{t18}, which is based on the
cycle iteration technique presented in Section \ref{reptcyc}.
\smallskip
\begin{theorem}\label{t20}
Given $\epsilon >0$, $\alpha \geq \beta$, and $y \in
\mathbb{T}^{N}$, $\mathbf{b} \in \mathcal{F}_{\alpha}(y) $ if and
only if
\begin{equation}\label{e57}
 \mathbb{E}_{k} \left  (\int_0^\tau L_{\omega(s)}(y + \mathcal I(\Xi)(s),-\Xi(s)) + \beta \, ds - b_{k}
+ b_{\omega(\tau)} \right )\geq 0,
\end{equation}
for any $k \in \{1, \cdots, m\}$, $\tau \gg \epsilon$ bounded  stopping times  and $\tau$--cycles $\Xi$.
\end{theorem}
We break the argument in two parts. The first one is  presented in a
preliminary lemma.

\smallskip
\begin{lemma}\label{preprese} Let $i \in \{1, \cdots, m\}$, $\b \in \mathbb R^m$, $\delta >0$,  assume $\tau^0$  to be a simple stopping time vanishing outside $\mathcal D_i$, with  $\tau^0 \gg \delta$  in $\mathcal D_i$ satisfying
\[\mathbb E_i \left  (\int_0^{\tau^0} L_{\omega(s)}(y + \mathcal
I(\Xi^0)(s),-\Xi^0(s)) + \beta \, ds - b_i + b_{\omega(\tau^0)} \right )
=: - \mu < 0. \] Then for any $j \in \mathbb N$
\begin{equation}\label{crux}
  \mathbb E_i \left  (\int_0^{\tau^j} L_{\omega(s)}(y + \mathcal I(\Xi^j)(s),-\Xi^j(s)) + \beta \, ds -
b_i + b_{\omega(\tau^j)} \right ) < - \mu \, ( 1+ \rho \, j ),
\end{equation}
where  $\tau^j$, $\Xi^j$ are as in \eqref{buckle1}, \eqref{cape1},
respectively, and $\rho$ is the positive constant, provided by
Proposition \ref{corollozzo}, with
\[ \left ( e^{- A \tau} \right )_{ik} > \rho \quad\hbox{for any $\tau \gg
\delta$ in $\mathcal D_i$, $k= 1, \cdots, m$.}\]
  \end{lemma}
\begin{proof}
 We denote by $I_j$  the expectation in the left hand--side of \eqref{crux} and argue by
  induction on $j$. Formula \eqref{crux} is true for
$j=0$, and we assume by inductive step that it holds for some $j
\geq 1$. Taking into account that
\[ \Xi^{j+1}(\omega)(s)= \Xi^j(\omega)(s) \qquad\hbox{in $[0,\tau^j(\omega))$ for
any $\omega$,}\] we get by applying the inductive step
\begin{equation}\label{crux0}
     I_{j+1} = I_j +  K_j \leq - \mu \, (
1+ \rho \, j ) + K_j
\end{equation}
 with
\[K_j= \mathbb E_i \left  (\int_{\tau^j}^{\tau^{j+1}} L_{\omega(s)}(y + \mathcal
I(\Xi^{j+1})(s),-\Xi^{j+1}(s)) + \beta \, ds -  b_{\omega(\tau^j)} +
b_{\omega(\tau^{j+1})} \right ).\]
   We further get by applying Lemma
\ref{precod} and the very definition of $\Xi^{j+1}$
\begin{equation}\label{crux1}
    K_j = \mathbb E_i\big( W(\omega)\big) +  \mathbb E_i  (- b_{\omega(\tau^j)} + b_{\omega(\tau^j +
\tau^0(\phi_{\tau^j}))})
\end{equation}
where
\[W(\omega) = \int_{\tau^j}^{\tau^j + \tau^0(\phi_{\tau^j})} L_{\omega(s)}(y + \mathcal
I(\Xi^0(\phi_{\tau^j}))(s - \tau^j),-\Xi^0(\phi_{\tau^j})(s -
\tau^j)) + \beta \, ds.\]
We fix $\omega$ and set $t =s - \tau^j(\omega)$, we
have
\[ W(\omega) =   \int_{0}^{\tau^0(\phi_{\tau^j}(\omega))}
L_{\phi_{\tau^j}(\omega)(t)}(y + \mathcal
I(\Xi^0(\phi_{\tau^j}(\omega)))(t),-\Xi^0(\phi_{\tau^j}(\omega))(t))
+ \beta \, dt. \] By using the above  relation and change of
variable  formula( from $\phi_{\tau^j}(\omega)$ to $\omega$), and
Theorem \ref{flowflow}, we obtain
\begin{equation}\label{crux2}
    \mathbb E_i W(\omega) = \mathbb E_{\mathbf{e}_i e^{-A \tau^j}} \left (\int_0^{\tau^0} L_\omega (y+ \mathcal
I(\Xi^0(\omega)), - \Xi^0(\omega))  + \beta \, ds \right )
\end{equation}
We also have by applying the same change of variable
\begin{eqnarray*}
  \mathbb E_i  \left (- b_{\omega(\tau^j)} + b_{\omega(\tau^j +
\tau^0(\phi_{\tau^j}))} \right ) &=&\mathbb E_i  \left (-
b_{\phi_{\tau^j}(\omega)(0)} + b_{\phi_{\tau^j}(\omega)(\tau^0)} \right ) \\
   &=& \mathbb E_{\mathbf{e}_i e^{-A \tau^j}}  \left (- b_{\omega(0)} + b_{\omega(\tau^0)}
\right ).
\end{eqnarray*}
By using the above relation, \eqref{crux1}, \eqref{crux2} and the
fact that $\tau^0$ vanishes outside $\mathcal D_i$, we obtain
\begin{eqnarray*}
  K_j &=& \mathbb E_{{\mathbf{e}_i e^{-A \tau^j}}} \left  (\int_0^{\tau^0} L_{\omega(s)}(y + \mathcal I(\Xi^0)(s),-\Xi^0(s)) + \beta \, ds -
b_i + b_{\omega(\tau^0)} \right) \\
&=& \left (e^{-A \tau^j}\right )_{ii} \, \mathbb E_i \left (\int_0^{\tau^0}
 L_{\omega(s)}(y + \mathcal I(\Xi^0)(s),-\Xi^0(s)) + \beta \, ds -
b_i + b_{\omega(\tau^0)} \right ) \\ &<& - \rho \, \mu
 \end{eqnarray*}
By plugging this relation in \eqref{crux0}, we end up with
\[I_{j+1}  \leq - \mu \, ( 1+ \rho \, (j+1) ) \]
proving \eqref{crux}.

 $\hfill{\Box}$
 \end{proof}

\medskip

\begin{proof} (of Theorem \ref{t20}) .\\
The first implication is direct by Theorem \ref{t18}.\\
Conversely, if $\mathbf b \notin \mathcal{F}_{\beta}(y)$ then there
exists, by Theorem \ref{t18}, $i \in \{1, \cdots, m\}$, bounded
stopping time $\tau^0$ and $\tau^0$--cycle $\Xi^0$ such that
\begin{equation} \label{pres1}
 \mathbb{E}_{i} \left  (\int_0^{\tau^0} L_{\omega(s)}(y + \mathcal I(\Xi^0)(s),-\Xi^0(s)) + \beta \, ds - b_{i}
+ b_{\omega(\tau^0)} \right ) =: - \mu  < 0.
\end{equation}
We can also assume $\tau^0 = 0$  outside $\mathcal D_i$ without
affecting \eqref{pres1}.  We set
 \[\widetilde \Xi(\omega)(s) = \left \{
\begin{array}{ll}
    \Xi^0(\omega)(s), & \hbox{for $\omega \in \mathcal{D}$, $s \in [0,\tau^{0}(\omega))$} \\
    0, & \hbox{for $\omega \in \mathcal D$, $s \in [\tau^0(\omega),+\infty)$.} \\
\end{array}
\right.\]
 We claim that $\widetilde \Xi$ is still a $\tau^0$--cycle; the unique property requiring
some detail is actually the nonanticipating character. We take $\omega_1=\omega_2$ in $[0,t]$,
for some positive $t$, and consider two possible cases:\\

\smallskip

\noindent \underline{\textbf{Case 1:}} If $s:=\tau^0(\omega_1)\leq
t$ then
$$\omega_1 \in A:=\{\omega \mid \tau^{0}(\omega)=s\} \in \F_s \subseteq \F_t,$$
which yields $\omega_2 \in A$ and hence $\tau^0(\omega_1)=\tau^0(\omega_2)=s$.\\
In this case
\begin{align*}
& \widetilde \Xi(\omega_1)= \Xi^0(\omega_1)= \Xi^0(\omega_2)=\widetilde\Xi(\omega_2)  \qquad\hbox{in $[0,s]$}, \\
&\widetilde\Xi(\omega_1)= \widetilde\Xi (\omega_2)=0
  \qquad\hbox{in $[s,t]$}.
\end{align*}
\smallskip
\noindent \underline{\textbf{Case 2:}} If $\tau^0(\omega_1)> t$,
then
$$\omega_1 \in \{\omega \mid \tau^{0}(\omega)>t\} \in \F_t ,$$
which implies that $\omega_{2}$ belongs to the above set and
consequently $\tau^0(\omega_2)> t$. Therefore
$$\widetilde\Xi(\omega_1)= \Xi^0(\omega_1)= \Xi^0(\omega_2)=\widetilde\Xi(\omega_2)  \qquad\hbox{in $[0,t]$}.$$
This shows the claim. Therefore we can assume that the $\Xi^0$
appearing in \eqref{pres1} vanishes when $t \geq \tau^0(\omega)$ for
any $\omega$.

 We know from Proposition \ref{stoppapproxi} that there is a nonincreasing sequence
$\tau'_n$ of simple stopping times with
\[  \tau'_n  \to \tau^0 \quad\hbox{uniformly in $\mathcal D$.}\]
We define
\[ \tau_n = \left \{\begin{array}{cc}
                       \tau'_n + \frac 1n & \;\hbox{in $\mathcal D_i$} \\
                      0 &  \;\quad\qquad\quad\;\hbox{in $\mathcal D_k$ for $k \neq i$} \\
                     \end{array} \right .\]
The $\tau_n$ are simple stopping times, moreover, since $\tau^0$ is
vanishing outside $\mathcal D_i$ and $\frac 1n \to 0$, we get
\[ \tau_n \geq \tau^ 0 \quad\hbox{and} \quad \tau_n  \to \tau^0 \quad\hbox{uniformly in $\mathcal D$,}\]
in addition
\[\{\tau_n \leq t\} \cap \mathcal D_i = \{\tau'_n  + 1/n \leq t\} \cap
\mathcal D_i \in \mathcal F_{t - 1/n} \quad\hbox{for $t \geq \frac {1}{n}$}\] which
shows that
\[\tau_n \gg \frac 1n \quad\hbox{in $\mathcal D_i$.}\]
It is clear that $\widetilde \Xi$ belongs to $\K(\tau_{n},0)$, we
further have
\begin{eqnarray*}
 \left | \int_0^{\tau_n} L_{\omega(s)}( y+\mathcal I(\Xi^0),- \Xi^0) \, ds
 -\int_0^{\tau^0} L_{\omega(s)}( y+\mathcal I(\Xi^0),-\Xi^0) \, ds  \right |\leq
 \int_{\tau^0}^{\tau_n }  |L_{\omega(s)}( y,0)| \, ds.
\end{eqnarray*}
Owing to the boundedness property of the integrand and that $\tau_n
\rightarrow \tau^0$, the right hand-side of the above formula
becomes infinitesimal, as $n$ goes to infinity.
Therefore the strict
negative inequality in (\ref{pres1}) is maintained replacing
  $\tau^0$
   by $\tau_n$, for a suitable $n$.

   Hence we can assume, without loss of generality,
   that $\tau^0$ appearing in \eqref{pres1} satisfies the assumptions of Lemma
\ref{preprese} for a suitable $\delta >0$.
Let $\tau^j$, $\Xi^j$ be as in \eqref{buckle1}, \eqref{cape1}, we
define for any $j$
\[  \widetilde\tau^j  = \tau^{j} + \epsilon \]
and
\[\widetilde \Xi^j(\omega)(s) = \left \{ \begin{array}{cc}
                      \Xi^{j}(\omega)(s) & \quad\hbox{for $s \in [0, \tau^{j}(\omega))$} \\
                      0 &   \quad\hbox {for $s \in [\tau^{j}(\omega),\widetilde\tau^{j}(\omega))$}
                    \end{array} \right .\]
the $\widetilde\tau^j$ are apparently  stopping times with
$\widetilde\tau^j \gg \epsilon$, and that $\widetilde \Xi^j$ are
$\widetilde\tau^j$--cycles. \\
We have
\[ \mathbb E_i \left  (\int_0^{\widetilde\tau^j} L_{\omega(s)}(y + \mathcal
I(\widetilde \Xi^j)(s),-\widetilde \Xi^j(s)) + \beta \, ds - b_i +
b_{\omega(\widetilde\tau^j)} \right )= U_j + V_j\] with
\begin{eqnarray*}
  U_j &=& \mathbb E_i \left (\int_0^{\tau^{j}} L_{\omega(s)}(y + \mathcal I(\Xi^{j})(s),-\Xi^{j}(s)) + \beta \, ds -
b_i + b_{\omega(\tau^{j})} \right ) \\
  V_j &=& \mathbb E_i \left  (\int_{\tau^{j}}^{\tau^{j} +\epsilon}
L_{\omega(s)}(y,0) + \beta  \, ds  - b_{\omega(\tau^{j})} +b_{\omega(\tau^{j}+
\epsilon)} \right )
\end{eqnarray*}
The term $U_j$ diverges negatively as $j \to + \infty$ by Lemma
\ref{preprese},  while $V_j$ stays bounded, which implies
\[ \mathbb E_i \left  (\int_0^{\widetilde\tau^j} L_{\omega(s)}(y + \mathcal
I(\widetilde \Xi^j)(s),-\widetilde \Xi^j(s)) + \beta \, ds - b_i +
b_{\omega(\widetilde\tau^j)} \right ) < 0\] for $j$ large. Taking
into account that $\widetilde \tau^j \gg \eps$ and Theorem
\ref{t18}, the last inequality shows that stopping times strongly
greater than $\eps$ and corresponding cycles based at $y$ are
sufficient to characterize values $\mathbf b \not\in \mathcal
F_\beta(y)$. This concludes the proof.

 $\hfill{\Box}$
\end{proof}

\medskip

   Next we state and prove  the first main theorem.

   \smallskip

 \begin{theorem}\label{t19}
 Given $\epsilon >0$,  $y \in \mathbb T^N$, $\mathbf b \in \R^m$, we consider
 \begin{equation}\label{e43}
  \inf \, \mathbb E_i
 \left [ \int_0^\tau L_{\omega(s)}(y +\mathcal I(\Xi),-\Xi) + \beta \, ds  - b_{\omega(0)}+b_{\omega(\tau)} \right ]
\end{equation}
where the infimum is taken with respect to any bounded stopping
times $\tau \gg \epsilon$ and $\tau$--cycles $ \Xi$.  The following
properties are equivalent:
\begin{itemize}
    \item[(i)] $y \in \mathcal A$
    \item[(ii)]  the infimum in \eqref{e43} is zero for any index $i$, any
    $ \mathbf b \in F_\beta(y)$
    \item[(iii)]  the infimum in \eqref{e43} is zero for some $i$, any
    $ \mathbf b \in F_\beta(y)$.
\end{itemize}
\end{theorem}
\smallskip

The assumption that the stopping times involved in the infimum are
strongly greater than a positive constant is essentially used for
proving (iii) $\Rightarrow$ (i), while in the implication (i)
$\Rightarrow$ (ii) it is exploited the characterization of
admissible values provided in Theorem \ref{t20}.

\smallskip

\begin{proof}

We start proving the implication (i) $\Rightarrow$ (ii). \\
Let  $y \in \mathcal{A}$, assume to the contrary that
$$ \inf \limits_{\tau \gg \epsilon} \, \mathbb E_{i}
 \left [ \int_0^\tau L_{\omega(s)}(y +\mathcal I(\Xi),-\Xi) + \beta \, ds - b_{\omega(0)}
+ b_{\omega(\tau)} \right ] \neq 0  \quad\hbox{for some $i$ and
$\mathbf b \in F_\beta(y)$.} $$ We deduce from Theorem \ref{t18}
that
\begin{equation} \label{e45}
 \inf \limits_{\tau \gg \epsilon} \, \mathbb E_{i}
 \left [ \int_0^\tau L_{\omega(s)}(y +\mathcal I(\Xi),-\Xi) + \beta \, ds - b_{\omega(0)}
+ b_{\omega(\tau)} \right ] > 0
\end{equation}
for such  $i$, $\mathbf b$. We claim that   $\mathbf b +\nu \,
\mathbf e_{i} \in \mathcal {F}_{\beta}(y)$ for any positive $\nu$
less than the infimum in \eqref{e45} denoted by $\eta$. Taking into
account (\ref{e45}) and $e^{-A \tau}$ being stochastic, we have for
any stopping time $\tau \gg \eps$ and $\tau$--cycle $\Xi$
\begin{eqnarray*}
   && \mathbb E_{i}\left [ \int_0^\tau L_{\omega(s)}(y +\mathcal I(\Xi),-\Xi) + \beta \, ds - (\mathbf b +\nu \,
   \mathbf e_{i})_{\omega(0)}+ (\mathbf b +\nu \, \mathbf e_{i})_{\omega(\tau)} \right ] \\
   &=& \mathbb E_{i}\left [ \int_0^\tau L_{\omega(s)}(y +\mathcal I(\Xi),-\Xi) + \beta \, ds - b_{\omega(0)}+ b_{\omega(\tau)} \right ]-\nu+\nu \, \mathbf e_{i}\, e^{-A \tau} \cdot \mathbf e_{i} \\
  &\geq& \eta -\nu \geq 0.
\end{eqnarray*}
We further get  for $j \neq i$ in force of  Theorem \ref{t18},
\begin{eqnarray*}
   && \mathbb E_{j}\left [ \int_0^\tau L_{\omega(s)}(y +\mathcal I(\Xi),-\Xi) + \beta \, ds - (\mathbf b +\nu \,
   \mathbf e_{i})_{\omega(0)}+ (\mathbf b +\nu \, \mathbf e_{i})_{\omega(\tau)} \right ] \\
   &=& \mathbb E_{j}\left [ \int_0^\tau L_{\omega(s)}(y +\mathcal I(\Xi),-\Xi) + \beta \, ds
   - b_{\omega(0)}+ b_{\omega(\tau)} \right ]-\nu \, \mathbf e_{j} \cdot \mathbf  e_{i}+\nu \, \mathbf e_{j}\,
   e^{-A \tau} \cdot \mathbf e_{i} \\
  &\geq& 0.
\end{eqnarray*}
Combining the information from the above computations with Theorem
\ref{t20}, we get the claim, reaching a contradiction with  $y$
being in $\mathcal{A}$ via Proposition \ref{prop28}.

\smallskip

It is trivial that (ii) implies (iii). We  complete the proof
showing  that (iii) implies  (i). Let us assume that \eqref{e43} is
vanishing  for some $i$ and any $\mathbf b \in F_\beta(y)$.

For any positive constant $\nu$, select $\delta > 0$ satisfying
$\delta - \rho \, \nu < 0$, where  $\rho > 0$ is given by
Proposition \ref{corollozzo}. Notice that we can invoke  Proposition
\ref{corollozzo}  because   we are working  with stopping times
strongly greater than $\eps$.   We fix $\mathbf b \in F_\beta(y)$
and deduce from \eqref{e43} being zero  that there exist a bounded
stopping time $\tau \gg \epsilon$, and  a $\tau$--cycle $ \Xi$ with
 \begin{equation} \label{e48}
  \mathbb E_{i}
 \left [ \int_0^\tau L_{\omega(s)}(y +\mathcal I(\Xi),-\Xi) + \beta \, ds - b_{\omega(0)}+b_{\omega(\tau)}\right ] <
 \delta.
 \end{equation}
 Taking into account Remark
\ref{corollozzobis} and \eqref{e48}, we have
 \begin{eqnarray*}
   && \mathbb E_{i}\left [ \int_0^\tau L_{\omega(s)}(y +\mathcal I(\Xi),-\Xi) + \beta \, ds - (b +\nu \, \mathbf  e_{i})_{\omega(0)}+ (b +\nu \, \mathbf e_{i})_{\omega(\tau)} \right ] \\
   &=& \mathbb E_{i}\left [ \int_0^\tau L_{\omega(s)}(y +\mathcal I(\Xi),-\Xi) + \beta \, ds - b_{\omega(0)}+ b_{\omega(\tau)} \right ]-\nu+\nu \, \mathbf  e_{i}\, e^{-A \tau} \cdot  \mathbf e_{i} \\
  &<& \delta - \rho \, \nu \\
  &<& 0
\end{eqnarray*}
which proves that $\mathbf{b}+\nu \, \mathbf e_{i} \notin \mathcal
{F}_{\beta}(y)$, in view of Theorem \ref{t18}.  This proves that
$\mathbf b$, arbitrarily taken in $F_\beta(y)$, is not an internal
point, and consequently that the interior of $F_\beta(y)$ must be
empty. This in turn implies that $y \in \mathcal A$ in force of
Proposition \ref{prop30}.
 $\hfill{\Box}$
\end{proof}

\medskip

Using expectation operators related to characteristic vectors of
stopping times, we have a more geometric formulation of the cycle
characterization provided in Theorem \ref{t19}, without any
reference to admissible values for critical subsolutions. This is
our second main result.

\smallskip

\begin{theorem}\label{t21}
 Given $\epsilon >0$, $y \in \mathcal{A}$ if and only if
 \begin{equation}\label{e43bis}
  \inf \, \mathbb E_{\mathbf{a}}
 \left [ \int_0^\tau L_{\omega(s)}(y +\mathcal I(\Xi),-\Xi) + \beta \, ds \right ]=0
\end{equation}
 where the infimum is taken with respect to any bounded stopping
times $\tau \gg \epsilon$, $\tau$--cycles $ \Xi$ and $\mathbf a$
characteristic vector of $\tau$.
\end{theorem}

\smallskip

Theorem \ref{t21} comes from Theorem \ref{t19} and the following

\smallskip

\begin{lemma}\label{post21} Given $\eps >0$, $y \in \mathcal A$ and $\mathbf b \in
F_\beta(y)$, let us consider
\begin{equation}\label{main1}
   \inf \, \mathbb E_i
 \left [ \int_0^\tau L_{\omega(s)}(y +\mathcal I(\Xi),-\Xi) + \beta \, ds  - b_{\omega(0)}+b_{\omega(\tau)} \right ]
\end{equation}
\begin{equation}\label{main2}
  \inf \, \mathbb E_{\mathbf{a}}
 \left [ \int_0^\tau L_{\omega(s)}(y +\mathcal I(\Xi),-\Xi) + \beta \, ds \right ]
\end{equation}
where both the infima are taken with respect to any bounded stopping
times $\tau \gg \epsilon$, $\tau$--cycles $ \Xi$, and  in
\eqref{main2} $\mathbf a$ is a characteristic vector of $\tau$.

Then \eqref{main2} vanishes if and only if  \eqref{main1} vanishes
for any $i \in \{1, \cdots, m\}$.
\end{lemma}
\begin{proof}
Let us  assume that \eqref{main1} vanishes for any $i$, then for any
$\delta >0$, any $i$  we find a $\tau_i \gg \eps$ and
$\tau_i$--cycles $\Xi_i$ with
$$ \mathbb E_i
 \left [ \int_0^{\tau_i} L_{\omega(s)}(y +\mathcal I(\Xi_i),-\Xi_i) + \beta \, ds  - b_{\omega(0)}+b_{\omega(\tau_i)}
 \right ] < \delta. $$
 We  define a new stopping time $\tau \gg \eps $ and a  $\tau$--cycle $\Xi$ setting
\begin{eqnarray*}
\tau &=&  \tau_i  \qquad\hbox{on $\mathcal D_i$}\\
\Xi  &=&  \Xi_i  \qquad\hbox{on $\mathcal D_i$}
\end{eqnarray*}
then we get $$ \mathbb E_i
 \left [ \int_0^{\tau } L_{\omega(s)}(y +\mathcal I(\Xi),-\Xi) + \beta \, ds  - b_{\omega(0)}+b_{\omega(\tau)}
 \right ] < \delta \qquad\hbox{for any $i$.} $$
 Taking  a characteristic vector $\a =(a_1, \cdots,
a_m)$ of $\tau$, and making convex combinations in the previous
formula with coefficients $a _i$, we  get  taking into account
Remark \ref{now} $$ \delta
> \mathbb E_{\mathbf a} \,\left [ \int_0^{\tau}  L_{\omega}(y+ \mathcal
I(\Xi),-\Xi)  + \beta \, ds \right ] - \mathbf a \cdot \mathbf b +
(\mathbf a \, e^{- A \tau}) \cdot \b = \mathbb E_{\mathbf a} \,\left
[ \int_0^{\tau} L_{\omega}(y+ \mathcal I(\Xi),-\Xi)  + \beta \, ds
\right ].$$ Since we know that the infimum in \eqref{main2} is
greater that or equal to $0$ thanks to Corollary \ref{cornow} with
$x=y$, the above inequality implies that it must be $0$, as claimed.

Conversely assume that \eqref{main2} is equal to $0$, then for any
$\delta >0$ there is a stopping time $\tau \gg \eps$ with
characteristic vector $\mathbf a$, and a $\tau$--cycle $\Xi$ with
$$ \mathbb E_{\mathbf{a}}
 \left [ \int_0^\tau L_{\omega(s)}(y +\mathcal I(\Xi),-\Xi) + \beta \, ds \right
 ] < \delta$$
Taking into account that $$\sum_i a_i \,\mathbb E_i [-b_{\omega(0)}
+ b_{\omega(\tau)}]= - \mathbf a \cdot \mathbf b
 +  (\mathbf a \, e^{-A \tau}) \cdot \mathbf b =0 $$
 for any $\mathbf b \in F_\beta(y)$, we derive
 $$ \sum_i a_i  \, \mathbb E_i
 \left [ \int_0^\tau L_{\omega(s)}(y +\mathcal I(\Xi),-\Xi) + \beta \, ds -  b_{\omega(0)}
 +
b_{\omega(\tau)} \right  ] < \delta.$$  From Remark \ref{min} and
the fact that the expectations in the above inequality must be
nonnegative because of Theorem \ref{t18}, we deduce
$$ \mathbb E_i \left [ \int_0^\tau L_{\omega(s)}(y +\mathcal I(\Xi),-\Xi) + \beta
\, ds -  b_{\omega(0)} + b_{\omega(\tau)} \right  ] < \frac \delta
\rho \qquad\hbox{for any $i$},$$ where $\rho$ is the constant
appearing in Proposition \ref{corollozzo}. This implies that the
infima in \eqref{main1} must vanish for any $i$.

$\hfill{\Box}$
\end{proof}

\medskip

\begin{appendices}
\section{Stochastic Matrices} \label{stocmatrx}

 \parskip +3pt

In this appendix we briefly collect some elementary linear algebraic results concerning stochastic matrices.
 The material is manly taken  from  \cite{Meyer}, \cite{Norris}, where  the reader can find  more details.

\medskip

 We denote by $\mathcal{S} \subset \mathbb{R}^{m}$ the simplex of probability
 vectors of $\mathbb{R}^{m}$, namely with nonnegative components summing to
 $1$. It is a compact convex set.
 \smallskip

 \begin{definition}
  A positive matrix $M$ is a matrix for which all the entries are positive, and we write $M > 0$.
\end{definition}

 \smallskip

 \begin{definition}
  A right stochastic matrix is a matrix of nonnegative entries with each row summing to $1$.
\end{definition}

\smallskip

\begin{proposition}\label{prop32} A matrix $B$ is stochastic if and only if
\begin{equation}\label{e53}
    \mathbf{a} \, B \in \mathcal{S} \;\; \mbox{whenever $\mathbf{a} \in \mathcal{S}$}.
\end{equation}
\end{proposition}
\begin{proof}
$B$ is stochastic if and only if each one of its rows
is a probability vector, i.e.
$$\mathbf e_{i} \,  B \in \mathcal{S} \quad\mbox{for every $i$},$$
 which in turn is equivalent to (\ref{e53}).

 $\hfill{\Box}$
\end{proof}

\smallskip

By Perron-Frobenius theorem for nonnegative matrices, we have
\begin{proposition}\label{prop33}
Let $B$ be a stochastic matrix, then its maximal  eigenvalue is $1$ and there is a corresponding left
eigenvector in $\mathcal{S}$.
\end{proposition}

\smallskip

By Perron-Frobenius theorem for positive matrices, we have

\begin{proposition}\label{prop34}  Let $B$ be a positive stochastic matrix, then its
maximal  eigenvalue is $1$ and is simple.  In addition, there exists a unique positive
corresponding left eigenvector which is an element of $\mathcal{S}$.
\end{proposition}

\medskip

In view of  application to  coupling matrices of  system, we recall

\smallskip

\begin{proposition} \label{prop26}
Given a matrix $A$ and $t\geq0$. Assume (A1) and (A2) hold, then $ e^{- At}$ is stochastic.
\end{proposition}

\smallskip

See  \cite{siconolfi3} for the proof.

\medskip

Exploiting the irreducibility condition (A3), we also have, see
Theorem 3.2.1 in \cite{Norris}.

\begin{proposition} \label{prop35}
Let $A$ be a matrix satisfying  (A1), (A2), (A3) then  $e^{-A t}$ is
positive for any $t >0$.
\end{proposition}


\section{Path spaces}\label{cadpath}
We refer readers to \cite{Pabi} for the material presented in this
section.

\bigskip

The term c\`{a}dl\`{a}g indicates a function, defined in some
interval of $\R$, which is continuous on the right and has left
limit.  We denote by $\mathcal D:=\mathcal D(0,+\infty;\{1, \cdots,
m\})$ and $\mathcal D(0,+\infty;\mathbb{R}^{N})$ the spaces of
c\`{a}dl\`{a}g paths defined in $[0,+\infty)$ with values in $\{1,
\cdots, m\}$ and $\mathbb{R}^{N}$, respectively.

\bigskip

To any finite increasing sequence of times $t_1, \cdots, t_k$, with
$k \in \mathbb{N}$,  and indices $j_1, \cdots, j_k$ in $\{1, \cdots,
m\}$ we associate a cylinder defined as
$$\mathcal {C}(t_1, \cdots,t_k;j_1, \cdots,j_k)= \{\omega \mid
\omega(t_1)=j_1, \cdots, \omega(t_k)=j_k\} \subset \mathcal{D}.$$
\smallskip
We denote by $\mathcal{D}_{i}$ cylinders of type $\mathcal C(0;i)$ for any $i \in \{1, \cdots, m\}$.\\
We call  multi-cylinders the sets made up by finite unions of mutually disjoint cylinders.

\bigskip

The space $\mathcal D$ of c\`{a}dl\`{a}g paths is endowed with the
$\sigma$--algebra $\mathcal F$
 spanned by cylinders of the type $\mathcal C(s;i)$, for $s \geq 0$ and $i \in \{1, \cdots, m\}$.
 A natural related filtration $\F_t$ is obtained by picking, as generating sets, the cylinders
$\mathcal C(t_1, \cdots,t_k;j_1, \cdots,j_k)$ with $t_k \leq t$, for
any fixed $t\geq 0$.

\medskip

We can perform same construction in $\mathcal D(0,+\infty;\mathbb{R}^{N})$, and in this case the $\sigma$--algebra,
 denoted by $\F'$,  is spanned by the sets
 \begin{equation}\label{cyli}
  \{\xi \in \mathcal D(0,+\infty;\mathbb{R}^{N}) \mid \xi(s) \in E\}
\end{equation}
for $s \geq 0$ and $E$ varying in the Borel $\sigma$--algebra
related to the natural topology of $\mathbb{R}^N$. A related
filtration is given by the increasing family of $\sigma$--algebras
$\F'_t$  spanned by cylinders in \eqref{cyli} with $s \leq t$.

\medskip

Both $\mathcal D$ and  $\mathcal D(0,+\infty;\mathbb{R}^{N})$ can be endowed with a metric, named after Skorohod,
 which makes them Polish spaces, namely complete and separable.
 Above $\sigma$--algebras $\F$, $\F'$ are the corresponding Borel $\sigma$--algebras.
\bigskip

The convergence induced by Skorohod metric is defined, say in
$\mathcal D(0,+\infty;\mathbb{R}^{N})$ to fix ideas, requiring that
there exists a sequence $f_n$ of  strictly increasing continuous
functions from $[0,+\infty]$  onto itself (then $f_n(0)=0$ for any
$n$)  such that
\begin{eqnarray*}
  f_n(s) &\to& s \quad\hbox{uniformly in $[0,+\infty]$} \\
  \xi_n(f_n(s)) &\to& \xi(s) \quad\hbox{uniformly in $[0,+\infty]$.}
\end{eqnarray*}

\bigskip

We consider the measurable shift flow $\phi_h$ on $\mathcal D$,
for $h \geq 0$, defined by

$$    \phi_h(\omega)(s)= \omega(s+h) \qquad\hbox{for any $s \in [0,+
\infty)$, $\omega \in \mathcal D$.}$$
\begin{proposition}\label{supershift} Given nonnegative constants $h$,
$t$, we have
\[ \phi_h^{-1}(\F_t) \subset \F_{t+h}.\]
\end{proposition}

\medskip

We also consider that space $\mathcal C(0,+\infty;\mathbb{T}^{N})$
of continuous paths defined in $[0,+ \infty)$  with the local
uniform convergence. We can associate to it a metric  making it a
Polish space.

We define a  map
$$\mathcal I: \mathcal D(0,+\infty;\mathbb{R}^{N}) \rightarrow  \mathcal C(0,+\infty;\mathbb{T}^{N})$$
via
$$\mathcal {I}(\xi)(t) = \mathrm{proj} \left ( \int_0^t \xi ds \right ) $$
where $\mathrm{proj}$ indicates the projection from $\R^N$ onto
$\mathbb T^N$. It is continuous with respect to the aforementioned
metrics, see \cite{siconolfi3}.

\section{Random setting}\label{randomset}

The material of this section is taken from \cite{siconolfi3}. We are
going to define a family of probability measures on $(\mathcal D,
\F)$, see \cite{siconolfi3}. We start from a preliminary result.
Taking into account that $\F$, $\F_t$ are generated by cylinders, we
get by the Approximation Theorem for Measures, see \cite[Theorem
1.65]{Kl}.

\smallskip

\begin{proposition}\label{klenke}
Let $\mu$ be a finite measure on $\F$. For any $E \in \F$, there is
a sequence $E_n$ of multi--cylinders  in $\F$ with
\[\lim_n \mu(E_n  \triangle E)=0, \]
where $\triangle$ stands for the symmetric difference. As a
consequence we see  that two finite measures on $\mathcal D$
coinciding on the family of cylinders, are actually equal.
\end{proposition}

\medskip
   Given a probability vector $\mathbf a$ in $\mathbb{R}^{m}$, namely with nonnegative components summing to 1,
    we define for any cylinder $\mathcal{C}(t_1, \cdots,t_k;j_1, \cdots , j_k)$ a nonnegative function $\mu_{a}$
\begin{equation}\label{e54}
    \mu_\a (\mathcal C(t_1, \cdots,t_k;j_1, \cdots,j_k)) = \left (\mathbf{a} \,e^{- A t_1} \right )_{j_1} \,
\prod_{l=2}^{k} \left (e^{-(t_l-t_{l-1})A} \right )_{j_{l-1}\,j_l}.
\end{equation}
We then exploit that $e^{-A t}$ is  stochastic to uniquely extend $\mu_{\mathbf
a}$, through Daniell--Kolmogorov Theorem, to a probability measure $\mathbb P_{\mathbf a}$  on $(\mathcal D, \mathcal F)$, see
for instance \cite[Theorem 1.2]{Swart}.\\
Hence, in view of (\ref{e54}), we have
\begin{proposition}\label{conti} The map
   $$\mathbf{a} \rightarrow \mathbb P_{\mathbf a}$$
    is injective, linear and continuous from $\mathcal S \subset \mathbb R^m$ to the space
   of probability measures on $\mathcal D$ endowed with the weak
   convergence.\\
Consequently, the measures $\mathbb P_{\mathbf a}$ are spanned by $\mathbb{P}_{i} :=\mathbb{P}_{\mathbf{e_{i}}}$, for $i \in \{1,\cdots,m\}$, and
\begin{equation}\label{prlin}
\mathbb{P}_{\mathbf{a}}= \sum_{i=1}^{m}a_{i} \, \mathbb{P}_{i}.
\end{equation}
\end{proposition}
By (\ref{e54}) we also get that $\mathbb{P}_{i}$ are supported in $\mathcal {D}_{i}:=\mathcal{C}(0;i)$.

\smallskip

We denote by $\mathbb{E}_\mathbf{a}$ the expectation operators relative to $\mathbb{P}_{\mathbf{a}}$, and we put $\mathbb{E}_{i}$ instead of $\mathbb{E}_{\mathbf{e_{i}}}$.\\
We say that some property holds almost surely, a.s. for short, if it
is valid up to $\mathbb{P}_{\mathbf{a}}-$null set for some
$\mathbf{a}
> 0$. We state for later use:
\smallskip

\begin{lemma}\label{nullo} Let $f$, $\mathbf{a}$ be a real random variable and a positive probability vector, respectively. If
\[\int_E f \, d\mathbb{P}_{\mathbf{a}} =0 \qquad\hbox{for any $E \in \mathcal F$}\]
then $f =0$ a.s.
\end{lemma}

\smallskip

We consider the push--forward of the probability measure $\mathbb
P_\mathbf{a}$, for any $\mathbf{a} \in \mathcal S$, through  the
flow $\phi_h$ on $\mathcal D$. In view of (\ref{e54}), one gets:
\begin{proposition} \label{pushprbdet}
For any $\mathbf{a} \in \mathcal S$, $h \geq 0$,
\[ \phi_h \# \mathbb P_\mathbf{a} = \mathbb P_{\mathbf{a} \, e^{-hA}}.\]
\end{proposition}

\smallskip

Accordingly, for any measurable function $f: \mathcal D \to \mathbb R$, we have by
the change of variable formula
\begin{equation}\label{changevar}
   \mathbb E_\mathbf{a} f(\phi_h)  = \int_\mathcal D f(\phi_h(\omega))\, d \mathbb P_\mathbf{a}= \int_\mathcal D f(\omega) \, d \phi_h \#
\mathbb P_\mathbf{a} =  \mathbb E_{\mathbf{a} \, e^{-A h}} f.
\end{equation}

\smallskip

The  push--forward  of $\mathbb{P}_{\mathbf{a}}$ through
$\omega(t)$, which is a random variable for any $t$, is a
probability measure on indices. More precisely, we have by
(\ref{e54})
$$\omega(t) \# \mathbb{P}_{\mathbf{a}}(i)=  \mathbb{P}_{\mathbf{a}}(\{\omega \mid \omega(t) = i\}) = \left (\mathbf{a} \, e^{-A t} \right )_i$$ for
any index $i \in \{1,\cdots,m\}$, so that
\begin{equation}\label{e55}
\omega(t) \#  \mathbb{P}_{\mathbf{a}} =  \mathbf{a}\, e^{-A t}.
\end{equation}
Moreover for $\mathbf b =(b_1, \cdots, b_m) \in \mathbb{R}^m$, we have
$$\mathbb E_{\mathbf a} b_{\omega (t)}= \mathbf a \, e^{-At} \cdot \mathbf b.$$

\medskip

Formula (\ref{e55}) can be partially recovered  for measures of the
type $ \mathbb{P}_{\mathbf{a}} \mres E$ which means $
\mathbb{P}_{\mathbf{a}}$ is restricted to $E$, where $E$ is any set
in $\mathcal{F}$.

\smallskip

 \begin{lemma}\label{lem4} (\cite{siconolfi3} Lemma 3.4) \,  For a given $ \mathbf{a} \in \mathcal {S}$, $E \in \mathcal{F}_{t}$ for
 some $t \geq 0$, we have
 $$  \omega(s) \# (\mathbb{P}_{\mathbf{a}} \mres E) = \big (  \omega(t) \# ( \mathbb{P}_{\mathbf{a}} \mres E) \big )  \,   e^{-A(s-t) } \quad \mbox{ for
    any $s \geq t$}.$$
\end{lemma}

\medskip

\noindent \textbf{Admissible controls:} We call control any random variable $\Xi$  taking values in $\mathcal D(0,+\infty;\mathbb{R}^{N})$ such that

\begin{itemize}
    \item[(i)] it is locally bounded (in time), i.e. for any $t >0$ there is
    $M >0$ with
    $$
    \sup_{[0,t]} |\Xi(t)| < M \qquad\mbox{ a.s.}
$$
    \item[(ii)] it is nonanticipating, i.e. for any $t >0$
    $$\omega_1=\omega_2 \;\mbox{in $[0,t]$} \;\; \Rightarrow  \;\; \Xi(\omega_1)=\Xi(\omega_2) \;\hbox{in
    $[0,t]$.}$$

The second condition is equivalent to require $\Xi$ to be adapted to the filtration $\mathcal{F}_{t}$ which means that the map $$\omega \mapsto \Xi(\omega)(t)$$
from $\mathcal{D}$ to $\mathbb{R}^{N}$ is measurable with respect $\mathcal{F}_{t}$ and the Borel $\sigma$--algebra on $\mathbb{R}^{N}$.
\end{itemize}
We will denote by $\mathcal K$ the class of admissible controls.

   \medskip
\noindent \textbf{Stopping times:}

   \begin{definition} A stopping time, adapted to $\mathcal F_{t}$, is a nonnegative random variable $\tau$ satisfying
$$\{ \tau \leq t\}   \in \mathcal F_t \qquad\mbox{for any $t$,} $$
which also implies $ \{ \tau < t\}, \, \{ \tau = t\}  \in \mathcal
F_t$.
\end{definition}
\smallskip

We will repeatedly use the following non increasing approximation of
a bounded random variable
 $\tau$ by simple stopping times. We set
\begin{equation}\label{splstopping}
    \tau_n = \sum_j \frac j{2^n} \, \mathbb{I}(\{\tau \in [(j-1)/2^n,j/2^n)\}),
\end{equation}
where $\mathbb I(\cdot)$ stands for the {\em indicator function} of
the set at the argument. We have for any $j$, $n$
\[\{\tau_n = j/2^n\} =\{\tau < j/2^n\} \cap \{\tau \geq
(j-1)/2^n\} \in \F_{j/2^n},\]  moreover the sum in
(\ref{splstopping})  is finite, being $\tau$ bounded. Hence $\tau_n$
are simple stopping times and letting $n$ go to infinity we get:
\smallskip

\begin{proposition} \label{stoppapproxi} Given a bounded stopping time
$\tau$, the $\tau_n$,  defined as in (\ref{splstopping}), make up a
sequence of simple stopping times with
\[
\tau_n \geq \tau, \quad \tau_n \to \tau  \quad\hbox{uniformly in
$\mathcal D$.}
\]
\end{proposition}
\smallskip

Given  a bounded stopping time $\tau$ and  a pair $x$, $y$ of
elements of $\mathbb{T}^{N}$, we set
  $$\mathcal {K}(\tau,y-x)=
  \left \{ \Xi \in \mathcal{K} \mid \mathcal I(\Xi)(\tau)  = y-x \, \mbox{a.s.} \right \}, $$
where the symbol $-$ refers to the structure of additive group on
$\mathbb T^N$ induced by the  projection of $\R^N$ onto $\mathbb T^N
= \R^N/\Z^N$.   The controls belonging to $\mathcal {K}(\tau,0)$ are
called $\tau$--cycles.
\end{appendices}


\end{document}